\documentclass{amsart}

\usepackage{amsmath, mathrsfs, amssymb,amsthm,hyperref,tikz,tikz-3dplot,centernot,comment,stmaryrd,enumerate}
\usepackage{tikz-cd}
\hypersetup{pdfstartview={XYZ null null 1.25}}
\usepackage[all]{xy}
\usepackage[foot]{amsaddr}

\newcommand{\pu}{p^\uparrow}

\newcommand{\bv}{\bigvee}
\newcommand{\bw}{\bigwedge}
\newcommand{\bN}{\mathbb N}

\theoremstyle{plain}
\newtheorem{thm}{Theorem}[section]
\newtheorem{prop}[thm]{Proposition}
\newtheorem{cor}[thm]{Corollary}
\newtheorem{lemma}[thm]{Lemma}
\newtheorem{ex}[thm]{Example}
\theoremstyle{definition}
\newtheorem{defn}[thm]{Definition}

\title{Recursive axiomatizations for representable posets}
\author{Rob Egrot}
\address{Faculty of ICT, Mahidol University, 999 Phutthamonthon Sai 4 Rd, Salaya, Nakhon Pathom 73170, Thailand}
\email{robert.egr@mahidol.ac.th}
\keywords{Poset, partially ordered set, poset representation, recursive axiomatization, axioms from games}
\subjclass[2010]{03C98, 06A11}

\date{}
\begin{document}

\begin{abstract}
We use model theoretic techniques and games to construct explicit first-order axiomatizations for the classes of posets that can be represented as systems of sets, where the order relation is given by inclusion, and existing meets and joins of specified countable cardinalities correspond to intersections and unions respectively.      
\end{abstract}

\maketitle

\section{Introduction}

Assuming the Axiom of Choice, or the weaker Prime Ideal Theorem, a lattice is isomorphic to a sublattice of a powerset algebra if and only if it is distributive. If $S$ is a (meet) semilattice, it is always possible to embed $S$ into a powerset algebra via a semilattice embedding preserving arbitrary meets. For example, we can just let $X$ be the set of all down-sets of $S$, considered as a complete lattice, and take $h:S\to X$ defined by $h(a)=a^\downarrow$, where $a^\downarrow=\{b\in S:b\leq a\}$.

The situation is more interesting if, in addition to the semilattice structure, we also wish to preserve some or all of the joins existing in $S$. Balbes \cite{Bal69} showed that, given $n$ with $3\leq n\leq \omega$, a semilattice embedding of $S$ into a powerset algebra that also preserves existing joins of cardinality strictly less than $n$ exists if and only if, for all $k<n$ and $x, y_1,\ldots ,y_k\in S$, whenever 
\[x\wedge (y_1\vee\ldots\vee y_k)\]
exists in $S$,  so too does
\[(x\wedge y_1)\vee\ldots\vee(x\wedge y_k),\]
and the two are equal. This result again requires some form of choice. Call a semilattice with this property $n$-representable. Clearly $n$-representable implies $m$-representable whenever $m\leq n$, but, unlike in the lattice case, the converse does not hold \cite{Kea97}. 

We can generalize to posets as follows. A poset $P$ is $(\alpha,\beta)$-representable if and only if it can be embedded into a powerset algebra via an embedding preserving meets of cardinality strictly less than $\alpha$, and joins of cardinality strictly less than $\beta$. For countable $\alpha$ and $\beta$, the class of $(\alpha,\beta)$-representable posets is known to be elementary \cite[Theorem 4.5]{Egr16}, but the proof does not produce an explicit axiomatization. It is known that the class has a finite axiomatization only in trivial cases \cite{Egr18}. 

In the uncountable case, the class will not be elementary at all, except if either $\alpha$ or $\beta$ is equal to 2, in which case it is just the class of all posets \cite{Egr17a}. In some cases however the associated class is pseudoelementary. A summary of what is known can be found in \cite[figure 2]{Egr17a}.

In this paper we construct explicit first-order axiomatizations for the cases where $\alpha$ and $\beta$ are countable. The idea is to define a two player game such that a winning strategy for a particular player corresponds to the existence of the required representation for a given poset, then to write down first-order sentences corresponding to the existence of such a strategy. We define a game for the case where $\alpha$ and $\beta$ are both equal to 3 in Section \ref{S:game}, and produce the associated axioms in Section \ref{S:ax}. We generalize to other choices of $\alpha$ and $\beta$ in Section \ref{S:gen}. 

To prove the resulting axiomatizations work for uncountable posets we must exploit the fact that first-order axiomatizations are already known to exist. This neat trick has appeared several times in the literature. For a recent example, and a result similar in spirit to the one proved here, see \cite[Section 4]{HirMcL17}, where it is used to produce a recursive axiomatization for representable disjoint union partial algebras. 

The full result as described above requires some form of choice. Section \ref{S:LSC} expands on this, and provides a countable analogue that does not require anything beyond ZF. Also in this section we explain how the known cases of lattices and semilattices relate to the general result for posets.   

\section{$(\alpha,\beta)$-representable posets}\label{S:PosRep}
This section summarizes some basic definition and results. We assume some knowledge of order theory and model theory, in particular ultraproduct constructions. Textbook exposition can be found in \cite{DavPri02} and \cite{Hodg93}. 

Let $\alpha$ and $\beta$ be cardinals greater than or equal to 2.

\begin{defn}
If $P$ and $Q$ are posets, then a monotone map $f:P\to Q$ is an \textup{\textbf{$(\alpha,\beta)$-morphism}} if: 
\begin{enumerate}[(1)]
\item if $S\subseteq P$ with $|S|< \alpha$ and $\bw S$ defined in $P$, then $h(\bw S)= \bigwedge h[S]$, and
\item if $T\subseteq P$ with $|T|< \beta$ and $\bv T$ defined in $P$, then $h(\bv T)= \bigvee h[T]$.
\end{enumerate}
When an $(\alpha,\beta)$-morphism is an order embedding we say it is an \textup{\textbf{$(\alpha,\beta)$-embedding}}. Note that if at least one of $\alpha$ and $\beta$ is strictly greater than $2$ then monotonicity is automatic. If $\alpha=\beta$ we sometimes just say, e.g. $\alpha$-\emph{morphism}. 
\end{defn}

\begin{defn}
A poset is \textup{\textbf{$(\alpha,\beta)$-representable}} if there is a set $X$ and an $(\alpha,\beta)$-embedding $h:P\to \wp(X)$ (where $\wp(X)$ is the power set of $X$, ordered by inclusion and considered as a lattice). If $\alpha=\beta$ we may just say, e.g. $\alpha$-\emph{representable}.
\end{defn}

\begin{defn}
An \textup{\textbf{$(\alpha,\beta)$-filter}} of $P$ is an upset $\Gamma\subseteq P$ such that:
\begin{enumerate}[(1)]
\item if $S\subseteq \Gamma$ with $|S|< \alpha$ and $\bw S$ defined in $P$, then $\bw S\in\Gamma$, and
\item if $T\subseteq P$ with $|T|< \beta$ and $\bv T$ defined in $P$, then $T\cap \Gamma=\emptyset\implies \bv T\notin \Gamma$.
\end{enumerate}
If $\alpha=\beta$ we can just say, e.g. $\alpha$-\emph{filter}.
\end{defn}

\begin{thm}\label{T:rep}
A poset $P$ is $(\alpha,\beta)$-representable if and only if for every $p\not\leq q \in P$ there is an $(\alpha,\beta)$-filter $\Gamma$ with $p\in\Gamma$ and $q\notin \Gamma$.
\end{thm}
\begin{proof}
If $P$ is $(\alpha,\beta)$-representable, by $h:P\to \wp(X)$ say, then $\{h^{-1}(x):x\in X\}$ is a separating set of $(\alpha,\beta)$-filters. Conversely, a separating set of $(\alpha,\beta)$-filters provides the base of a representation defined by 
\[h(p)=\{\Gamma\in \wp(P) :\Gamma\text{ is an $(\alpha,\beta)$-filter and }p\in \Gamma\}.\]
\end{proof}

A dual version of this result holds for a suitable concept of ideals. See \cite{Egr16} for a more thorough discussion.

\begin{prop}\label{P:closed}
Let $2\leq\alpha,\beta\leq\omega$. Then the class of $(\alpha,\beta)$-representable posets is closed under the following constructions:
\begin{enumerate}[(1)]
\item Products.\label{prod}
\item Ultraproducts.\label{Uprod}
\item Ultraroots.\label{Uroot}
\end{enumerate}
\end{prop}
\begin{proof}
\begin{enumerate}[(1)]
\item Let $I$ be an indexing set and suppose $P_i$ is an $(\alpha,\beta)$-representable poset for all $i\in I$. Let $x,y\in \prod_I P_i$ and suppose $x\not\leq y$. Then there is $j\in I$ such that $x(j)\not\leq y(j)$. By Theorem \ref{T:rep} there is an $(\alpha,\beta)$-filter, $\Gamma_j$, with $x(j)\in \Gamma_j$ and $y(j)\notin\Gamma_j$. For $i\in I\setminus\{j\}$ define $S_i= P_i$, and define $S_j =\Gamma_j$. It is straightforward to check that $\prod_I S_i$ is an $(\alpha,\beta)$-filter of $\prod_I P_i$ containing $x$ and not $y$, and so the result follows from Theorem \ref{T:rep}. Note that this proof works for all $2\leq\alpha,\beta$, so $\alpha$ and $\beta$ need not be countable. 

\item This is \cite[Proposition 4.2]{Egr16}, but we provide an alternative proof here. Given an indexing set $I$, and $(\alpha,\beta)$-representable posets $P_i$ for each $i\in I$, we let $U$ be a non-trivial ultrafilter of $\wp(I)$ and construct the ultraproduct $\prod_U P_i$. Now, if $[a]$ and $[b]$ are elements of $\prod_U P_i$ with $[a]\not\leq[b]$, then $\{i\in I: a(i)\not\leq b(i)\}$ is equal to $u$ for some $u\in U$. For each $j\in u$ let $\Gamma_j$ be an $(\alpha,\beta)$-filter of $P_j$ with $a(j)\in \Gamma_j$ and $b(j)\notin \Gamma_j$ (using Theorem \ref{T:rep}). We extend the signature of posets by a single unary predicate $G$, and we interpret $G$ in $P_j$ by
\[P_j\models G(x)\iff x\in \Gamma_j\]
when $j\in u$. So, by definition of the ultraproduct, $\prod_U P_i\models G([a])$ and $\prod_U P_i\models \neg G([b])$. Now, for all $j\in u$, the fact that $\{x\in P_j: G(x)\}$ is an $(\alpha,\beta)$-filter can be expressed in first-order logic using the signature of posets plus $G$, so, by \L o\'s Theorem \cite{Los55}, $\{[x]\in\prod_U P_i: G([x])\}$ is also an $(\alpha,\beta)$-filter, and the result follows from Theorem \ref{T:rep}.  

\item This is \cite[Proposition 4.4]{Egr16}, and we can prove it quickly by observing that if $p\not\leq q\in P$, and if $[\bar{p}]$ and $[\bar{q}]$ are the elements of the ultrapower $\prod_U P$ defined by the constant sequences, $\bar{p}(i) = p$ for all $i\in I$, and $\bar{q}(i) = q$ for all $i\in I$ respectively, then an $(\alpha,\beta)$-filter of $\prod_U P$ containing $[\bar{p}]$ and not $[\bar{q}]$ restricts to an $(\alpha,\beta)$-filter of the canonical image of $P$ in $\prod_U P$.  
\end{enumerate}
\end{proof}

\begin{thm}\label{T:elem}
For all $2\leq\alpha,\beta\leq\omega$, the class of $(\alpha,\beta)$-representable posets is elementary.
\end{thm}
\begin{proof}
This is \cite[Theorem 4.5]{Egr16}. The class is obviously closed under isomorphism, and since it is closed under ultraproducts and ultraroots, by Proposition \ref{P:closed} parts (2) and (3) respectively, it is elementary by the Keisler-Shelah Theorem \cite{Kei61,She71}.
\end{proof}

\begin{thm}\label{T:fin}
Given $2\leq\alpha,\beta\leq\omega$, the class of $(\alpha,\beta)$-representable posets is finitely axiomatizable if and only if at least one of $\alpha$ and $\beta$ is $2$.
\end{thm}
\begin{proof}
If both $\alpha$ and $\beta$ are strictly greater than $2$, then that the class is not finitely axiomatizable is the main result of \cite{Egr18}. Conversely, if $\beta=2$, then given $p\not\leq q\in P$ we can always take $\pu$ as an $(\alpha,\beta)$-filter containing $p$ but not $q$, and, if $\alpha=2$, we can take $P\setminus\{q^\downarrow\}$ for the same purpose. Thus in these cases every poset is $(\alpha,\beta)$-representable, by Theorem \ref{T:rep}. 
\end{proof}

\begin{prop}
The class of $(\alpha,\beta)$-representable posets is not closed under substructures when both $\alpha$ and $\beta$ are strictly greater than $2$.
\end{prop}
\begin{proof}
See Example \ref{E:subs} below.
\end{proof}

\begin{cor}
The class of $(\alpha,\beta)$-representable posets has no universal axiomatization when both $\alpha$ and $\beta$ are strictly greater than $2$.
\end{cor}
\begin{proof}
A class is universal if and only if it is closed under isomorphisms, ultraproducts and substructures (see e.g. \cite[Theorem 2.20]{BurSan81}).
\end{proof}

Note that, by the easy part of the proof of Theorem \ref{T:fin}, if one or both of $\alpha$ and $\beta$ is $2$, then the class of $(\alpha,\beta)$-representable posets is just the class of all posets, which is universally axiomatized by definition. 

\begin{ex}\label{E:subs}
Let $P$ be the poset in figure \ref{F:subs}. Then it's straightforward to check that $P$ is $(\alpha,\beta)$-representable, by using Theorem \ref{T:rep}, or just by writing down a suitable representation. However, the substructure with base $\{a,b,c,\top,\bot\}$ is order isomorphic to the diamond lattice $M_3$, and so is not distributive. 
\end{ex}

\begin{figure}[htbp]
\centering
\scalebox{0.9}{\xymatrix{ & \bullet_\top\ar@{-}[d]   \\
& \bullet\ar@{-}[d]   \\
& \bullet\ar@{..}[d]  \\
& & &\\
\bullet_a\ar@{..}[ur] & \bullet_b\ar@{..}[u] & \bullet_c\ar@{..}[ul]\\
& \bullet_\bot\ar@{-}[ul]\ar@{-}[u]\ar@{-}[ur] 
}} 
\caption{}
\label{F:subs}
\end{figure}

\section{A game for $3$-representations}\label{S:game}
We define a family of two player games, played over a poset $P$, between players $\forall$ and $\exists$. The games are played in rounds indexed by the natural numbers. In each round, $\forall$ plays a move, then $\exists$ plays a move. The choices $\exists$ makes add elements to a set $U$ whose initial state is dependent on the game. $\exists$ wins an $n$-round game if  $\forall$ does not win till at least round $n+1$. $\exists$ wins an $\omega$-round game if $\forall$ does not win at any point during play. The idea is that, given $p\not\leq q\in P$, and a suitable sequence of moves by $\forall$, a winning play by $\exists$ in one of these games constructs a $3$-filter containing $p$ but not $q$. 

Our games are played as follows. Every game has a \emph{starting position}, taking the form of a pair of sets $U$ and $V$. In every round of the game $\forall$ makes one of the following moves (for $a,b\in P$):
\begin{enumerate}
\item If $b\geq a$, for some $a\in U$, then $\forall$ can play $(b)$.
\item If $a,b\in U$, and $a\wedge b$ is defined in $P$, then $\forall$ can play $(a,b)$.
\item If $a\vee b$ is defined and in $U$, then $\forall$ can play $(a,b)$.
\end{enumerate}

$\exists$ must respond to each move according to the following rules (one for each possible move by $\forall$):

\begin{enumerate}
\item $\exists$ must  add $b$ to $U$.
\item $\exists$ must add $a\wedge b$ to $U$.
\item $\exists$ must choose either $a$ or $b$ and then add it to $U$.
\end{enumerate}

$\forall$ wins in round $n$ if $U\cap V\neq \emptyset$ at the beginning of that round. $\exists$ has an $n$-strategy for the game with starting position $(U,V)$ if she can guarantee that she will not lose this game till at least round $n+1$, however $\forall$ plays. $\exists$ has an $\omega$-strategy for the game with starting position $(U,V)$ if she can guarantee to never lose this game, whatever moves $\forall$ makes. Note that if $U\cap V\neq\emptyset$ in the starting position then $\exists$ will lose in round 0.  

\begin{prop}\label{P:game}
If $P$ is $3$-representable then, for all $p\not\leq q\in P$, $\exists$ has an $\omega$-strategy in the game for $P$ with starting position $(\{p\},\{q\})$. If $P$ is countable then the converse is true. 
\end{prop}
\begin{proof}
If $P$ is $3$-representable then, whenever $p\not\leq q\in P$, there is a $3$-filter $\Gamma$ containing $p$ and not $q$ (by Theorem \ref{T:rep}). In this case $\exists$ has an $\omega$-strategy where she only picks elements from $\Gamma$. 

Conversely, suppose $P$ is countable, and suppose that $\exists$ has an $\omega$-strategy for the game with starting position $(\{p\},\{q\})$, whenever $p\not\leq q\in P$. We define a strategy for $\forall$ as follows. Fix an arbitrary 1-1 map $v:P\times P\to \bN$. We say a move $\forall$ can make is \emph{good} if one of the following holds:
\begin{enumerate}
\item The move is type 1 and $b\notin U$.
\item The move is of type 2 and $a\wedge b\notin U$.
\item The move is of type 3 and neither $a$ nor $b$ is in $U$.
\end{enumerate}
We rank the moves that could potentially be played by $\forall$ using the following system: 
\begin{enumerate} 
\item If $(b)$ is a move of type 1 then the rank of $(b)$ is $v(b,b)$.  
\item If $(a,b)$ is a move of type 2 or 3 then the rank of $(a,b)$ is $v(a,b)$.
\end{enumerate}
In each round $\forall$ plays the \emph{good} move with the lowest rank. If no \emph{good} move exists then $\forall$ just plays the type 1 move $(p)$ for the rest of the game. Then the set $U$ constructed during this game is a $3$-filter, because, once a \emph{good} move has become available for $\forall$, either he plays it within a finite number of moves, or it stops being a \emph{good} move after a finite number of rounds. 

Only \emph{good} moves are relevant to the question of whether $U$ is a $3$-filter, so during the course of this game, if $\exists$ plays according to her $\omega$-strategy, every potential obstacle to $U$ being a $3$-filter is removed. Since $U$ contains $p$ and not $q$ by definition, after an appeal to Theorem \ref{T:rep} we are done.
\end{proof}

\section{Axioms for winning games}\label{S:ax} 
We can easily use the signature of posets to write down (universal) formulas $J$ and $M$, with free variables $x$, $y$ and $z$, such that an interpretation of $J(x,y,z)$ holds in $P$ if and only if the interpretation of $z$ is the join of the interpretations of $x$ and $y$, and an interpretation of $M(x,y,z)$ holds in $P$ if and only if the interpretation of $z$ is the meet of the interpretations of $x$ and $y$. 

Given $k\in\omega$, we can also write down a (quantifier free) formula $C_k$, with free variables $(x_1,\ldots,x_k,y)$, such that an interpretation of $C_k(x_1,\ldots,x_k,y)$ holds in $P$ if and only if the interpretation of $y$ is equal to the interpretation of $x_i$ for some $i\in\{1,\ldots,k\}$.  Similarly, we can write down a (quantifier free) formula $D_k$ such that an interpretation of $D_k(x_1,\ldots,x_k,y)$ holds in $P$ if and only if the interpretation of $y$ is distinct from the interpretation of $x_i$ for all $i\in\{1,\ldots,k\}$. So $C_0$ just defines the empty set, and $D_0$ is satisfied by all elements, though this is not significant in the proof. Note also that $D_k = \neg C_k$.

For all $m,n\in\omega$ we use $\vec{x}_m$ to denote the $m$-tuple $(x_1,\dots,x_m)$, and we define formulas $\phi_{mn}$, with $m+1$ free variables, by recursion as follows.

\[\phi_{m0}(\vec{x}_m,y) = D_m(\vec{x}_m,y).\] 
\begin{alignat*}{2}&\phi_{m(n+1)}(\vec{x}_m,y)=\\ 
&\phantom{=} \forall ab\Big(\big((\exists c (C_m(\vec{x}_m,c)\wedge(c\leq a))\rightarrow \phi_{(m+1)n}(\vec{x}_m,a,y)\big)\\
&\phantom{=}\wedge\big((C_m(\vec{x}_m,a)\wedge C_m(\vec{x}_m,b)\wedge\exists cM(a,b,c))\rightarrow \phi_{(m+1)n}(\vec{x}_m,c,y)\big)\\
&\phantom{=}\wedge \big(\exists c(C_m(\vec{x}_m,c)\wedge J(a,b,c))\rightarrow  (\phi_{(m+1)n}(\vec{x}_m,a,y)\vee \phi_{(m+1)n}(\vec{x}_m,b,y))\big)\Big).
\end{alignat*}

\begin{lemma}\label{L:ngames}
Let $P$ be a poset, and let $v$ be an assignment of variables to elements of $P$. Then, for all $m,n\in\omega$, 
\begin{align*}P, v\models \phi_{mn}(\vec{x}_m,y)\iff&\text{ $\exists$ has an $n$-strategy for the game with}\\ &\text{starting position }(\{v(x_1),\ldots,v(x_m)\},\{v(y)\}).
\end{align*}
\end{lemma}
\begin{proof}
Given a tuple of variables $\vec{x}_m=(x_1\ldots,x_m)$, if we abuse notation slightly and define $v[\vec{x}_m]=\{v(x_1),\ldots,v(x_m)\}$, then, by definition, \[P,v\models \phi_{m0}(\vec{x}_m,y)\iff v[\vec{x}_m]\cap \{v(y)\}=\emptyset,\] which is equivalent to saying that $\exists$ has a 0-strategy in the game with starting position $(v[\vec{x}_m],\{v(y)\})$. The result then follows by inspection of the formula $\phi_{mn}$ and induction on $n$. To see this note that, applying the inductive hypothesis, we have $P,v \models \phi_{m(n+1)}(\vec{x}_m,y)$ if and only if, whatever move $\forall$ makes first in the game with starting position $(v[\vec{x}_m],\{v(y)\})$, $\exists$ can respond with a move $(a)$ in such a way that she has a winning strategy for the $n$-round game with starting position $(v[\vec{x}_m]\cup\{a\},\{v(y)\})$. But this is equivalent to saying that $\exists$ has a strategy for the $(n+1)$-round game with starting position $(v[\vec{x}_m],\{v(y)\})$. In other words, $P,v \models \phi_{m(n+1)}(\vec{x}_m,y)$ if and only if $\exists$ has an $(n+1)$-strategy for the game with starting position $(v[\vec{x}_m],\{v(y)\})$, as required.  
\end{proof}

Now, for each $n\in \omega$ we define sentences as follows.
\[\psi_n = \forall xy((x\not\leq y) \rightarrow \phi_{1n}(x,y)).\]

\begin{prop}\label{P:ax}
\begin{align*}P\models \psi_n \iff&\text{ for all $p,q\in P$, if $p\not\leq q$ then $\exists$ has an $n$-strategy}\\ &\text{for the game with starting position $(\{p\},\{q\})$.}\end{align*}
\end{prop}
\begin{proof}
This follows directly from Lemma \ref{L:ngames}.
\end{proof}

\begin{thm}\label{T:models}
Let $P$ be a poset. Then $P$ has a $3$-representation if and only if $P\models \psi_n$ for all $n\in \omega$.
\end{thm}
\begin{proof}
If $P$ is $3$-representable then it has a separating set of $3$-filters, by Theorem \ref{T:rep}. So, given $p\not\leq q\in P$, $\exists$ can pick a $3$-filter, $\Gamma$, containing $p$ but not $q$, and win the $\omega$-game with starting position $(\{p\},\{q\})$ just by always picking elements of $\Gamma$. That $P\models \psi_n$ for all $n\in\omega$ then follows directly from Proposition \ref{P:ax}.  

For the converse, suppose first that $P$ is countable, and that $P\models \psi_n$ for all $n\in\omega$. Then, by Proposition \ref{P:ax}, whenever $p\not\leq q\in P$, $\exists$ has an $n$-strategy for the game with starting position $(\{p\},\{q\})$. It follows from K\"onig's Tree Lemma \cite{Kon26} that $\exists$ has an $\omega$-strategy for the game with starting position $(\{p\},\{q\})$. Thus $P$ is $3$-representable by Proposition \ref{P:game}.

To complete the proof we use a trick. If $P$ is uncountable then it has a countable elementary substructure, $P'$, by the downward L\"owenheim-Skolem Theorem. Since $P'\models \psi_n$ for all $n\in\omega$, we have just shown that $P'$ must be 3-representable. But the class of 3-representable posets is elementary, by Theorem \ref{T:elem}. So, as $P$ and $P'$ are elementarily equivalent, it follows that $P$ is $3$-representable too. 
\end{proof}

The axiomatization we have generated for the class of $3$-representable posets is clearly recursively enumerable, and so by Craig's trick  there is also a recursive axiomatization for this class. Note that Craig's trick should not be confused with the Interpolation Theorem (see e.g. \cite[Exercise 6.3.1]{Hodg93} for the trick we use here, in case it is unclear).

\section{Generalizing}\label{S:gen}
We can generalize the approach taken in Sections \ref{S:game} and \ref{S:ax} to $(\alpha,\beta)$-representations, for all $2\leq \alpha,\beta\leq\omega$. This requires a straightforward modification of the game rules.

\begin{defn}[$(\alpha,\beta)$-game]
Let $P$ be a poset, let $U,V\subseteq P$, and let $2\leq \alpha,\beta\leq\omega$. The $(\alpha,\beta)$-game with starting position $(U,V)$ is defined similarly to the game defined in Section \ref{S:game}. $\forall$ has the following choices:
\begin{enumerate}
\item If $b\geq a$, for some $a\in U$, then $\forall$ can play $(b)$.
\item If $A\subseteq U$ with $|A|<\alpha$, and if $\bw A$ is defined in $P$, then $\forall$ can play $A$.
\item If $B\subseteq P$ with $|B|<\beta$, and if $\bv B$ is defined in $P$ and is in $U$, then $\forall$ can play $B$.
\end{enumerate}

$\exists$ must respond according to the following rules:
\begin{enumerate}
\item $\exists$ must  add $b$ to $U$.
\item $\exists$ must add $\bw A$ to $U$.
\item $\exists$ must choose some $b\in B$ and add it to $U$.
\end{enumerate}

\end{defn}

\begin{prop}
Let $P$ be a poset, and let $2\leq \alpha,\beta\leq\omega$. Then, if $P$ is $(\alpha,\beta)$-representable, $\exists$ has an $\omega$-strategy in the $(\alpha,\beta)$-game with starting position $(\{p\},\{q\})$ for all $p\not\leq q\in P$. Moreover, if $P$ is countable then the converse is true.
\end{prop}
\begin{proof}
The proof of Proposition \ref{P:game} can easily be adapted. In the case where either $\alpha$ or $\beta$ is 2 then $P$ is always $(\alpha,\beta)$-representable, and consequently $\exists$ always has an $\omega$-strategy for the necessary starting positions.  
\end{proof}

Given $2\leq \alpha,\beta <\omega$, we can follow the approach of Section \ref{S:ax} and write down first-order sentences equivalent to $\exists$ having an $n$-strategy in the $(\alpha,\beta)$-game with starting position $(\{p\},\{q\})$ for all $p\not\leq q\in P$. A result analogous to Theorem \ref{T:models} can be proved for these using the same techniques as before, though a little more care is required when defining the axiomatizing first-order theory. The case where one or both of $\alpha,\beta$ is $\omega$ requires a slight modification of the approach.

We extend the notation from Section \ref{S:ax} as follows. 

\begin{itemize}
\item For all $1\leq k<\omega$ let $J_k(\vec{x}_k,y)$ and $M_k(\vec{x}_k,y)$ be first-order formulas that hold in a poset $P$ under assignment $v$ if and only if $v(y)=\bv v[\vec{x}_k]$ and $v(y)=\bw v[\vec{x}_k]$ respectively.
\item For all $1\leq k ,m <\omega$, let $C_{km}(\vec{x}_k,\vec{y}_m)$ be a first-order formula that holds in $P$ under assignment $v$ if and only if $v[\vec{y}_m]\subseteq v[\vec{x}_k]$.
\item For all $1\leq k ,m <\omega$, let $D_{km}(\vec{x}_k,\vec{y}_m)$ be a first-order formula that holds in $P$ under assignment $v$ if and only if $v[\vec{y}_m]\cap v[\vec{x}_k]=\emptyset$.
\end{itemize} 

For all $1\leq k, r, s<\omega$ define the following formulas.

\[\sigma_{k}(\vec{x}_k,c) = \exists z \big(C_k(\vec{x}_k,z)\wedge(z\leq c)\big).\]
\[\tau_{kr}(\vec{x}_k,\vec{a}_r,c) = C_{kr}(\vec{x}_k,\vec{a}_r)\wedge M_r(\vec{a}_r,c).\]
\[\rho_{ks}(\vec{x}_k,\vec{b}_s) = \exists z\big(C_k(\vec{x}_k,z)\wedge J_s(\vec{b}_s,z)\big).\]

Now, for all $1\leq k,r,s<\omega$ and for all $n\in\omega$ we define formulas $\phi_{krsn}$ using recursion.

\[\phi_{krs0}(\vec{x}_k,y) = D_k(\vec{x}_k,y).\] 

\begin{alignat*}{2}&\phi_{krs(n+1)}(\vec{x}_k,y)=\\ 
&\phantom{=} \forall \vec{a}_r\forall \vec{b}_s\forall c\Big(\big(\sigma_{k}(\vec{x}_k,c)\rightarrow \phi_{(k+1)rsn}(\vec{x}_k,c,y)\big)\\
&\phantom{=}\wedge\big(\tau_{kr}(\vec{x}_k,\vec{a}_r,c)\rightarrow \phi_{(k+1)rsn}(\vec{x}_k,c,y)\big)\\
&\phantom{=}\wedge \big(\rho_{ks}(\vec{x}_k,\vec{b}_s)\rightarrow  \bv_{i=1}^s \phi_{(k+1)rsn}(\vec{x}_k,b_i,y)\big)\Big).
\end{alignat*}

\begin{lemma}\label{L:ngames2}
Let $P$ be a poset, and let $v$ be an assignment of variables to elements of $P$. Then, for all $n\in\omega$ and $1\leq k, r,s <\omega$, 
\begin{align*}P, v\models \phi_{krsn}(\vec{x}_k,y)\iff&\text{ $\exists$ has an $n$-strategy for the $(\alpha,\beta)$-game with}\\ &\text{starting position }(v[\vec{x}_k],\{v(y)\})\\& \text{for all $2\leq\alpha \leq r+1$ and $2\leq \beta\leq s+1$}.
\end{align*}
\end{lemma}
\begin{proof}
This is essentially the same as the proof for Lemma \ref{L:ngames}. 
\end{proof}

Now, for all $n\in\omega$, and for all $1\leq r,s<\omega$, define
\[\psi_{rsn} = \forall xy((x\not\leq y) \rightarrow \phi_{1rsn}(x,y)).\]

\begin{prop}\label{P:ax2}
Let $1\leq\alpha,\beta<\omega$.
\begin{align*}P\models \psi_{\alpha\beta n} \iff&\text{ for all $p,q\in P$, if $p\not\leq q$ then $\exists$ has an $n$-strategy}\\ &\text{for the $(\alpha+1,\beta+1)$-game with starting position $(\{p\},\{q\})$.}\end{align*}
\end{prop}
\begin{proof}
This follows directly from Lemma \ref{L:ngames2}.
\end{proof}

\begin{thm}\label{T:models2}
Let $P$ be a poset, and let $2\leq\alpha,\beta<\omega$. Then $P$ has an $(\alpha,\beta)$-representation if and only if $P\models \psi_{(\alpha-1)(\beta-1) n}$ for all $n\in \omega$.
\end{thm}
\begin{proof}
This is essentially the same as the proof of Theorem \ref{T:models}.
\end{proof}
If we allow $\alpha$ and $\beta$ to take the value $\omega$ then we must modify the approach slightly, and use more axioms.
\begin{thm}\label{T:models3}
Let $P$ be a poset, and let $2\leq\alpha,\beta\leq\omega$. Then $P$ has an $(\alpha,\beta)$-representation if and only if $P\models \psi_{rs n}$ for all $n\in \omega$ and for all $1\leq r < \alpha$ and $1\leq s <\beta$.
\end{thm}
\begin{proof}
If $P$ is $(\alpha,\beta)$-representable then $\exists$ has an $\omega$-strategy for all the necessary games, as in the proof of Theorem \ref{T:models}, and so one direction of the result follows from Proposition \ref{P:ax2}.

For the converse, we follow the pattern of the proof of Theorem \ref{T:models}, and assume first that $P$ is countable. By Proposition \ref{P:ax2} and K\"onig's Tree Lemma we know that, for all $2\leq r< \alpha$ and $2\leq s<\beta$, $\exists$ has an $\omega$-strategy for the $(r,s)$-game with starting position $(\{p\},\{q\})$ whenever $p\not\leq q\in P$.

It follows that $\exists$ has an $\omega$-strategy for the $(\alpha,\beta)$-game with starting position $(\{p\},\{q\})$ whenever $p\not\leq q\in P$. This is also a consequence of the Tree Lemma. If $\exists$ does not have an $\omega$-strategy for some $p\not\leq q$, then there must be a bounded game tree for this starting position, and so $\forall$ has a strategy that forces a win within a finite number of rounds. But then the size of sets he plays in this winning strategy must be bounded. It follows that there are $l<\alpha$ and $k<\beta$ such that $\forall$ can force a win in the $(l,k)$-game with starting position $(\{p\},\{q\})$. But this is a contradiction as we have already proved that $\exists$ has an $\omega$-strategy in this game. 

From here the proof proceeds exactly like that of Theorem \ref{T:models}.  
\end{proof}

\section{Lattices, semilattices and choice}\label{S:LSC}
As discussed in the introduction, the cases where $P$ is a lattice or a semilattice are well understood. Using the terminology of Section \ref{S:PosRep}, a lattice is $(\alpha,\beta)$-representable for all $3\leq \alpha,\beta\leq\omega$ if and only if it is distributive, and a meet semilattice is $(\alpha,\beta)$-representable for given $2\leq \beta \leq \omega$ and all $2\leq \alpha\leq\omega$ if and only if it is $\beta$-distributive as defined in Definition \ref{D:dist} below.

It is natural to ask how the known axiomatizations for distributive lattices and $n$-distributive semilattices relate to the axiomatizations produced using the game technique described here. Proposition \ref{P:lattice} provides an answer to this.

\begin{defn}\label{D:dist}
Let $S$ be a meet semilattice, and let $2\leq \beta \leq\omega$. Then $S$ is \textup{\textbf{$\beta$-distributive}} if, for all $m<\beta$, and for all and $x, y_1,\ldots ,y_m\in S$, whenever 
\[x\wedge (y_1\vee\ldots\vee y_m)\]
exists in $S$,  so too does
\[(x\wedge y_1)\vee\ldots\vee(x\wedge y_m),\]
and the two are equal.  
\end{defn}

\begin{prop}\label{P:lattice}
Let $S$ be a meet semilattice, and let $2\leq k\leq\omega$. Then the following are equivalent:
\begin{enumerate}[(1)]
\item $\exists$ has an $\omega$-strategy in the $(m,k)$-game with starting position $(\{a\},\{b\})$ for all $a\not\leq b\in S$, and for all $2\leq m\leq \omega$.
\item $\exists$ has a 5-strategy in the $(3,k)$-game with starting position $(\{a\},\{b\})$ for all $a\not\leq b\in S$.
\item $S$ is $k$-distributive.
\end{enumerate}
\end{prop}
\begin{proof}\mbox{}
\begin{enumerate}
\item[](1)$\implies$(2): Automatic.

\item[](2)$\implies$(3): If $S$ is not $k$-distributive then $k>2$, and there must be $m< k$ and $x,y_1,\ldots,y_m\in S$ such that $x\wedge (y_1\vee\ldots\vee y_m)$ is defined and, either $(x\wedge y_1)\vee\ldots\vee(x\wedge y_m)$ is not defined, or \[(x\wedge y_1)\vee\ldots\vee(x\wedge y_m)< x\wedge (y_1\vee\ldots\vee y_m).\]

In either case there must be $z\in P$ with $z\geq (x\wedge y_i)$ for all $i\in\{1,\ldots,m\}$ and $x\wedge (y_1\vee\ldots\vee y_m)\not\leq z$. Thus $\forall$ can force a win in the game with starting position $(\{x\wedge (y_1\vee\ldots\vee y_m)\},\{z\})$ by the start of the 5th round by playing according to the following strategy. First he plays the type 1 moves $(x)$, and $(y_1\vee\ldots\vee y_m)$, and $\exists$ must respond by adding these elements to $U$. Then $\forall$ plays the type 3 move $\{y_1,\ldots,y_m\}$, and $\exists$ must respond by adding $y_i$ to $U$ for some $i$. Then $\forall$ plays the type 2 move $\{x,y_i\}$, and $\exists$ must respond by adding $x\wedge y_i$ to $U$. Finally, $\forall$ plays the type 1 move $(z)$, which forces $\exists$ to add $z$ to $U$, and thus $\forall$ wins at start of the 5th round. 

\item[](3)$\implies$(1): Let $S$ be $k$-distributive, and let $a\not\leq b\in S$. Suppose $\exists$ does not have an $\omega$-strategy for the game with starting position $(\{a\},\{b\})$. Then, in particular there must be no $(\omega,k)$-filter containing $a$ and not $b$, as otherwise this filter would provide an $\omega$-strategy for $\exists$. But a version of the Prime Ideal Theorem for semilattices is a consequence of the Axiom of Choice (AC) when $S$ is $k$-distributive (see \cite[Theorem 2.2]{Bal69}), and so it follows that AC is false. Thus we have disproved AC in ZFC, which is of course impossible if ZFC is consistent. 
\end{enumerate}     
\end{proof}

Proposition \ref{P:lattice} and Theorem \ref{T:models2} produce a finite axiomatization for the class of $(\omega, \beta)$-representable meet semilattices whenever $2\leq\beta<\omega$. Note that while we have shown that a meet semilattice is $\omega$-distributive if and only if $\exists$ has a 5-strategy in the associated $(3,\omega)$-games, this does not produce a finite axiomatization for the class of $(\omega, \omega)$-representable semilattices, which we know does not exist \cite{Kea97}. This is because we must use Theorem \ref{T:models3}, and here an infinite number of axioms are needed even to characterize the existence of a 5-strategy for $\exists$.   

\begin{cor}
Let $L$ be a lattice. Then the following are equivalent:
\begin{enumerate}[(1)]
\item $\exists$ has an $\omega$-strategy in the $(m,k)$-game with starting position $(\{a\},\{b\})$ for all $a\not\leq b\in L$, and for all $3\leq m,k\leq \omega$.
\item $\exists$ has a 5-strategy in the $(3,3)$-game with starting position $(\{a\},\{b\})$ for all $a\not\leq b\in L$.
\item $L$ is distributive.
\end{enumerate}
\end{cor}
\begin{proof}
This follows straightforwardly from Proposition \ref{P:lattice}.
\end{proof}

Note that we refer to AC in the proof of Proposition \ref{P:lattice}, but the proof does not depend on its assumption, so this proposition is a theorem of ZF. 

Of course, in general, to deduce the existence of an $\omega$-strategy from the existence of an $n$-strategy for all $n$ we require some choice principle, as this is necessary in the proof of K\"onig's Tree Lemma. The axiom of countable choice for finite sets, for example, is always sufficient. In the proof of Theorem \ref{T:models}, the Tree Lemma is applied to a countable set, so the necessary version is a theorem of ZF (indeed, the full power of ZF is not needed - see \cite[Chapter III.7]{Sim09}). However, to complete the proof of Theorem \ref{T:models} we rely implicitly on the Keisler-Shelah Theorem.

The following result however can be proved without any choice.

\begin{prop}
Let $2\leq \alpha,\beta\leq \omega$. Then there is a recursive set of first-order sentences $\Theta$, such that whenever $P$ is a countable poset we have 
\[P\models \Theta\iff P\text{ is }(\alpha,\beta)\text{-representable}.\]  
\end{prop}
\begin{proof}
Let $\Theta$ be recursive set of sentences used in either Theorem \ref{T:models2} or \ref{T:models3}, depending on whether either $\alpha$ or $\beta$ is $\omega$. We can proceed more or less as in the proofs of these theorems. The points of departure being that, as $P$ is countable, we only need the countable version of the Tree Lemma, which, as mentioned previously, is a theorem of ZF, and we don't need to use the L\"owenheim-Skolem + Keisler-Shelah trick.  
\end{proof}

\section*{Acknowledgement}
The author would like to thank Robin Hirsch for inviting him to UCL for the visit during which this paper was written, and for various valuable discussions during his stay. The author would also like to thank the Department of Computer Science at UCL for hosting him.

%\bibliography{../../../references}{}
%\bibliographystyle{abbrv} 

\end{document}